\documentclass{amsart}
\usepackage{amsmath}
\usepackage{amsfonts}
\usepackage{amssymb}
\usepackage{ascmac}
\usepackage{layout}
\usepackage{amsthm}
\usepackage[pdftex]{graphicx}
\usepackage{tikz}

\makeatletter
\@namedef{subjclassname@2020}{%
  \textup{2020} Mathematics Subject Classification}
\makeatother
							  
\newcommand{\force}{\Vdash}

\title{The extent of saturation of induced ideals}

\author{Kenta Tsukuura}
\address{Doctoral Program in Mathematics, Degree Programs in Pure and Applied Sciences, Graduate School of Science and Technology, University of Tsukuba, Tsukuba, 305-8571, Japan}
\email{tukuura@math.tsukuba.ac.jp}

\thanks{This research was supported by Grant-in-Aid for JSPS Research Fellow Number 20J21103. The author is grateful to Masahiro Shioya for helpful discussions.}

\date{}

\theoremstyle{plain}
\newtheorem{thm}{Theorem}[section]
\newtheorem{defi}[thm]{Definition}
\newtheorem{lem}[thm]{Lemma}
\newtheorem{coro}[thm]{Corollary}
\newtheorem{ques}[thm]{Question}
\newtheorem{clam}[thm]{Claim}

\newtheorem{rema}[thm]{Remark}
\newtheorem{prop}[thm]{Proposition}

\begin{document}
\maketitle

\begin{abstract}
 We construct a model with a saturated ideal ${I}$ over $\mathcal{P}_{\kappa}\lambda$ and study the extent of saturation of $I$.
\end{abstract}
\section{Introduction}
The existence of a saturated ideal on a successor cardinal is a kind of generic large cardinal axioms. The first model with a saturated ideal on $\aleph_1$ was constructed by Kunen~\cite{MR495118}. He established
\begin{thm}[Kunen]\label{kunen}
 Suppose that $j$ is a huge embedding with critical point $\kappa$. Then there is a poset $P$ such that $P \ast \dot{S}(\kappa,j(\kappa))$ forces that $\aleph_1$ carries a saturated ideal.
\end{thm}
See Section 2 for the definitions of a saturated ideal and its strengthenings below.
 $S(\kappa,j(\kappa))$ denotes Silver collapse. The poset $P$, called the universal collapse, is useful in constructing a model with a saturated ideal because of its nice absorption property. Indeed, using the method of universal collapses each of the following was shown to be consistent:
\begin{enumerate}
 \item (Laver~\cite{MR673792}) $\aleph_1$ carries a strongly saturated ideal.
 \item (Foreman--Magidor--Shelah~\cite{MR942519}) $\aleph_1$ carries a layered ideal.
 \item (Foreman--Laver~\cite{MR925267}) $\aleph_1$ carries a centered ideal.
\end{enumerate}
Kunen's proof has been improved in two ways. One is due to Magidor. He used a sequence of local master conditions, while Kunen used a single master condition. Incorporating this improvement, Foreman--Magidor--Shelah proved (2), only assuming the existence of an almost-huge cardinal with target Mahlo, rather than a huge cardinal. The improvement enables us to use the Levy collapse instead of the Silver collapse, and weaken the assumption of Theorem \ref{kunen} to the existence of an almost-huge cardinal. The other improvement is due to Shioya~\cite{shioya}. He pointed out that the diagonal product of Silver collapses has a nice absorption property and works as $P$ in Theorem \ref{kunen}.

In this paper, we construct a model with a saturated ideal on the successor of a given regular cardinal, combining the two improvements. We also study the extent of saturation of our ideal. We will show that 
 \begin{thm}\label{maintheorem} Suppose that $j$ is an almost-huge embedding with critical point $\kappa$ and $\mu < \kappa \leq \lambda < j(\kappa)$ are regular cardinals. Then $P(\mu,\kappa) \ast \dot{\mathrm{Coll}}(\lambda,<j(\kappa))$ forces that there is a saturated ideal ${I}$ over $\mathcal{P}_\kappa\lambda$ with the following properties:
 \begin{enumerate}
  \item ${I}$ is $(\lambda^{+},\lambda^{+},<\mu)$-saturated.
  \item $I$ is not $(\lambda^{+},\mu,\mu)$-saturated. In particular, $I$ is not strongly saturated. 
  \item $I$ is layered if and only if $j(\kappa)$ is Mahlo in $V$.
  \item $I$ is not centered. In particular, $I$ is not strongly layered.
\end{enumerate}
\end{thm}
Here, $P(\mu,\kappa)$ is the diagonal product of Levy collapses. See Section 4 for the definition of $P(\mu,\kappa)$. Note that an ideal over $\kappa$ can be seen as an ideal over $\mathcal{P}_{\kappa}\kappa$. 

Modifying our proof of Theorem~\ref{maintheorem} we can show that
\begin{thm}
Let $I$ be the saturated ideal on $\aleph_1$ in the model of Theorem \ref{kunen}. Then the following hold:
\begin{enumerate}
 \item $I$ is $(\aleph_2,\aleph_2,<\aleph_0)$-saturated.
 \item $I$ is not $(\aleph_2,\aleph_0,\aleph_0)$-saturated. In particular, $I$ is not strongly saturated. 
 \item (Foreman--Magidor--Shelah~\cite{MR942519}) $I$ is layered.
 \item (Foreman--Laver~\cite{MR925267}) $I$ is not centered.
\end{enumerate}
\end{thm}
Foreman--Laver~\cite{MR925267} claimed (4) without proof.

The structure of this paper is as follows. In Section 2, we recall basic facts of forcing projections and some saturation properties. In Section 3, we introduce two properties of continuous and $(\lambda,\lambda,<\mu)$-nice for the projections. We also see these properties work when we study the saturation properties of quotient forcing. In Section 4, we introduce the diagonal product of Levy collapse $P(\mu,\kappa)$ and study the saturation property of $P(\mu,\kappa)$ and Levy collapses. In Section 5, we give the proof of Theorem \ref{maintheorem}.

 \section{Preliminaries}
In this section, we recall basic facts of forcing and saturated ideal. 
We use~\cite{MR1994835} as a reference for set theory in general. For more on the topic of saturated ideal, we refer to~\cite{MR2768692}.

Our notation is standard. We use $\kappa,\lambda$ to denote a regular cardinal unless otherwise stated. We also use $\mu$ to denote a cardinal, possibly finite, unless otherwise stated. For $\kappa < \lambda$, $E^{\lambda}_\kappa$, $E^{\lambda}_{>\kappa}$ and $E^{\lambda}_{\leq\kappa}$ denote the set of all ordinals below $\lambda$ of cofinality $\kappa$, $>\kappa$ and $\leq\kappa$, respectively. We also write $[\kappa,\lambda) = \{\xi \mid \kappa \leq \xi < \lambda\}$. By $\mathrm{Reg}$, we mean the class of regular cardinals.

Throughout this paper, we identify a poset $P$ with its separative quotient. Thus, $p \leq q\leftrightarrow \forall r \leq p(r {\parallel} q) \leftrightarrow p \force q \in \dot{G}$, where $\dot{G}$ is the canonical name of $(V,P)$-generic filter. A projection $\pi:Q \to P$ is an order-preserving mapping with the property that $q \leq_P \pi(p)$ implies $\exists r \leq_Q p(\pi(r) \leq_P q)$ and $\pi(1_P) = 1_{Q}$. We say that $P$ is a complete suborder of $Q$, denoted by $P \lessdot Q$ if the identity mapping $i:P \to Q$ is a complete embedding. 
Whenever a projection $\pi:Q \to P$ is given, for every dense $D$ in $P$, $\pi^{-1}D$ is also dense in $Q$. It follows that $Q \force \pi``\dot{H}$ generates a $(V,P)$-generic filter, where $\dot{H}$ is the canonical name of $(V,Q)$-generic filter. The quotient forcing is defined by $P \force Q / \dot{G} = \{q \in Q\mid \pi(q) \in \dot{G}\}$, ordered by $\leq_P$. In addition, we can define a dense embedding $\tau: Q \to P \ast Q / \dot{G}$ by $\tau(q) = \langle{\pi(q),\hat{q}}\rangle$ where $\hat{q}$ is a $P$-name with $P \force \pi(q) \in \dot{G} \to \hat{q} = q$ and $\pi(q) \not \in \dot{G} \to \hat{q} = 1$. Thus, $Q \simeq P \ast Q/\dot{G}$. 
The completion of $P$ is a complete Boolean algebra $B(P)$ such that $P \lessdot B(P)$ is a dense subset. $B(P)$ is unique up to isomorphism. 

Whenever a projection $\pi:Q \to P$ is given, there is a complete embedding $e:P \to \mathcal{B}(Q)$ which is defined by $e(p) = \sum\{q\mid \pi(q) \leq p\}$ . It is easy to see that $\pi(e(p)) = p$ and $e(\pi(q)) \geq q$. First, we check about basic properties of $e$ and $\pi$ for Section 3.
 
\begin{lem}
 If $e:P \to Q$ is a complete embedding between complete Boolean algebras, then the following holds:
\begin{enumerate}
 \item For every $A \subseteq P$, $e(\prod A) = \prod e``A$.
 \item If $e$ is defined by a projection $\pi:Q \to P$, $\pi(e(p) \cdot q) = p\cdot \pi(q)$.
\end{enumerate}
\end{lem}

\begin{proof}
 (1) It is easy to see $e(\prod A) \leq \prod e``A$. Let us see $\prod e``A \leq e(\prod A)$. By separativity, it suffices to show that $\forall b \leq \prod e``A(b \cdot e(\prod A) \not= 0)$. Let $c \in P$ be a reduct of $b$, that is $\forall d \leq c(e(d)\cdot b \not= 0)$. For every $d \leq c$ and $a \in A$, $e(d) \cdot e(a)\geq e(d) \cdot b \not = 0$. Thus, $d \cdot a \not= 0$ for all $d \leq c$ in $P$, especially $c \leq a$. Therefore, $c \leq \prod A$ and $b \cdot e(\prod A) \geq b \cdot e(c) \not= 0$ 

 (2) Observe that $\pi(e(p) \cdot q) \leq \pi(e(p)) \cdot \pi(q) = p \cdot \pi(q)$. To show $p \cdot \pi(q) \leq \pi(e(p) \cdot q)$, we check $\forall r \leq \pi(q) \cdot p(r \cdot (p \cdot \pi(q)) \not= 0)$. For any $r \leq \pi(q) \cdot p$, there is an $s \leq q$ with $\pi(s) \leq r$. By $q \cdot e(p) = q \cdot \sum\{x \mid \pi(x) \leq p\} = \sum\{q \cdot x \mid \pi(x) \leq p\}$, $s = q \cdot s \leq q \cdot e(p)$. Therefore $\pi(s) \leq \pi(q \cdot e(p)) \cdot r$. 
\end{proof}

Next, we define saturation properties that we will deal with in this paper. For cardinals $\mu\leq \kappa \leq \lambda$, we say that $P$ has the $(\lambda,\kappa,<\mu)$-c.c. if, for every $X \in [P]^{\lambda}$, there is a $Y \in [X]^{\kappa}$ such that $Z$ has a lower bound for all $Z \in [Y]^{<\mu}$. By $(\lambda,\kappa,\mu)$-c.c., we mean $(\lambda,\kappa,<\mu^{+})$-c.c. Of course, $\lambda$-c.c. and $\lambda$-Knaster are the same as $(\lambda,2,2)$-c.c. and $(\lambda,\lambda,2)$-c.c., respectively. 

For a stationary subset $S \subseteq \lambda$, we say that $P$ is $S$-layered if there is an $\subseteq$-increasing sequence $\langle P_\delta \mid \delta < \lambda \rangle$ with the following properties:
	\begin{enumerate}
	 \item $P = \bigcup_{\delta < \lambda}P_\delta$.
	 \item $P_{\delta} \lessdot P$ and $|P_\delta| < \lambda$ for all $\delta < \lambda$.
	 \item There is a club $C \subseteq \lambda$ such that $\forall \delta \in S \cap C(P_\delta = \bigcup_{\zeta < \delta}P_{\zeta})$. 
	\end{enumerate}
We say such sequence a $S$-layering of $P$.
For a later purpose, we introduce the notion of filtration. We say that $\langle P_{\delta} \mid \delta < \lambda \rangle$ is a filtration of $P$ if it is an $\subseteq$-increasing continuous sequence with $P = \bigcup_{\delta < \lambda}P_\delta$ and $|P_{\delta}| < \lambda$ for all $\delta < \lambda$. 
\begin{lem}\label{layeredchar}
 The following are equivalent.
 \begin{enumerate}
  \item $P$ is $S$-layered.
  \item There are a filtration $\langle P_{\delta} \mid \delta < \lambda \rangle$ and a club $C \subseteq \lambda$ such that $P_\delta \lessdot P$ for all $\delta \in S \cap C$. 
  \item For any filtration $\langle P_{\delta} \mid \delta < \lambda \rangle$, there is a club $C \subseteq \lambda$ such that $P_\delta \lessdot P$ for all $\delta \in S \cap C$. 
 \end{enumerate}
\end{lem}
\begin{proof}
 In \cite{MR3911105}, it is shown that (2) and (3) are equivalent. We check that (1) and (2) are equivalent. First, we assume (1). Let $\langle P_\alpha \mid \alpha < \lambda \rangle$ be a $S$-layering sequence of $P$. It is easy to see that $\langle \bigcup_{\beta < \alpha}P_{\beta} \mid \alpha < \lambda \rangle$ is a filtration which witnesses to (2). 

 Let us see the inverse direction. Let $\langle Q_{\alpha} \mid \alpha < \lambda \rangle$ be a filtration of $P$ and $C\subseteq \lambda$ be a club such that $Q_{\alpha} \lessdot P$ for all $\alpha \in S \cap C$. Define $P_{\alpha} = Q_{\min (S\cap C \setminus \alpha)}$ for all $\alpha < \lambda$. It is easy to see that $\langle P_{\alpha} \mid \alpha < \lambda \rangle$ is a $S$-layering sequence of $P$. 
\end{proof}

Shelah showed that $S$-layered implies the $\lambda$-c.c. Moreover, Cox~\cite{MR3911105} showed that $S$-layered implies $\lambda$-Knaster. 

 We say that $P$ is $\lambda$-centered if $P$ is a union of $\lambda$-many centered subset of $P$. A centered subset is $X \subseteq P$ such that $\forall Z \in [X]^{<\omega}(X$ has a lower bound$)$. We call such centered sets a centering of $P$. It is easy to see that every $\lambda$-centered poset has the $(\lambda^{+},\lambda^{+},<\omega)$-c.c., which in turn implies the $\lambda^{+}$-c.c. These properties are preserved by projection. Indeed,
\begin{lem}
 If $\pi:Q \to P$ is a projection. Then the following holds.
\begin{enumerate}
 \item If $Q$ has the $(\lambda,\lambda,<\mu)$-c.c., then so does $P$.
 \item If $Q$ is $S$-layered for some stationary $S \subseteq \lambda$, then so is $P$.
 \item If $Q$ is $\lambda$-centered, then so is $P$.
\end{enumerate} 
\end{lem}
It is easy to see that $(\lambda,\lambda,<\mu)$-c.c. and $\lambda$-centeredness are preserved under taking the completion. We need to be careful in the case of layeredness, but there is no harm in the current definition because the cardinalities remain the same after taking the completions in this paper.

In this paper, by ideal, we means a fine and normal ideal. 
For an ideal $I$ over $\mathcal{P}_{\kappa}\lambda$, $\mathcal{P}(\mathcal{P}_\kappa(\lambda)) / I$ denotes the poset $\mathcal{P}(\mathcal{P}_{\kappa}\lambda) \setminus I$ ordered by $A \leq B\leftrightarrow A \setminus B \in I$. We say that $I$ is $(\lambda',\kappa',<\mu')$-saturated if $\mathcal{P}(\mathcal{P}_{\kappa}\lambda) / I$ has the $(\lambda',\kappa',<\mu')$-c.c. Simply, we say $I$ is saturated if $I$ is $(\lambda^{+},2,2)$-saturated. 

Similarly, we say that $I$ is strongly saturated, layered and centered if $\mathcal{P}(\mathcal{P}_{\kappa}\lambda) / I$ is $(\lambda^{+},\lambda^{+},<\lambda)$-saturated, $S$-layered for some stationary subset $S \subseteq E^{\lambda^{+}}_{\lambda}$ and $\lambda$-centered, respectively.  The implications between these properties are as follows:

 
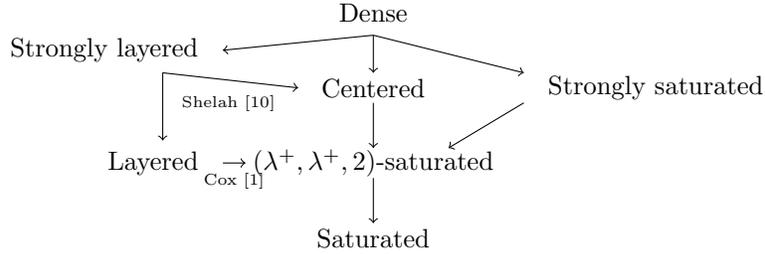
\begin{center}
 \begin{tikzpicture}
  \node at (0,0) {{Dense}};
  \draw[->] (0,-0.3) to (-2,-0.5);
  \draw[->] (0,-0.3) to (0,-0.8);
  \draw[->] (0,-0.3) to (2,-0.8);
  \node[anchor = east] at (-2.2,-0.5) {Strongly layered};
  \node[anchor = east] at (-2.2,-2) {Layered};
  \node at (0,-1) {Centered};
  \node[anchor = west] at (2.2,-1) {Strongly saturated};
  \draw[->] (0,-1.2) to (0,-1.8); 
  \draw[->] (2,-1.2) to (1,-1.8);
  \node at (0,-2) {$(\lambda^{+},\lambda^{+},2)$-saturated};
  \draw[->] (0,-2.2) to (0,-2.8);
  \node at (0,-3) {Saturated};

  \draw[->] (-2.8,-0.8) to (-2.8,-1.7);
  \draw[->] (-2,-2) -- node[auto = right]{\tiny{Cox~\cite{MR3911105}}} (-1.7,-2);
  \draw[->] (-2.8,-0.8) --node[pos = 0.9, auto = right]{\tiny{Shelah~\cite{MR850051}}} (-1,-1);
 \end{tikzpicture}
\end{center}
Here, $I$ is dense and strongly layered if $\mathcal{P}(\mathcal{P}_{\kappa}\lambda)/I$ has a dense subset of size $\lambda$ and $\mathcal{P}(\mathcal{P}_{\kappa}\lambda)/I$ is $E^{\lambda^{+}}_{\lambda}$-layered, respectively.

\section{Continuity of projections}
Our study of ideals will be reduced to that of the quotient forcing induced by projections. In the first half of this section, we give sufficient conditions for the quotient to have the desired properties.
\begin{defi}
 Suppose $\pi:Q \to P$ is a projection between complete Boolean algebras. We say that $\pi$ is $<\mu$-continuous if $\pi(\prod Z) = \prod \pi``Z$ for all $Z \in [Q]^{<\mu}$ with $\prod^{Q}Z\not= 0$.

We also say that $\pi$ is continuous if $\pi$ is $<\mu$-continuous for all $\mu$. 
\end{defi}
For a projection $\pi:Q \to P$ between posets, we say that $P$ is $<\mu$-continuous if the lifting $\pi:\mathcal{B}(Q) \to \mathcal{B}(P)$ is $<\mu$-continuous. For the following lemma, recall that a Boolean algebra $P$ is $(<\mu,\lambda)$-distributive if $P$ adds no new sequences to $\lambda$ of length $<\mu$. 
\begin{lem}\label{ccquotient}
  Suppose that $P$ is $(<\mu,\lambda)$-distributive, $\pi:Q \to P$ is $<\mu$-continuous projection between complete Boolean algebras, and $Q$ has the $(\lambda,\lambda,<\mu)$-c.c. then $P \force Q/ \dot{G}$ has the $(\lambda,\lambda,<\mu)$-c.c.
\end{lem}
\begin{proof}
 Let $p$ and $\{\dot{q}_{\alpha} \mid \alpha < \lambda \}$ be arbitrary with $p \force \dot{q}_{\alpha} \in Q / \dot{G}$. For each $\alpha$, we can take $r_{\alpha}$ such that $\pi(r_\alpha) \leq p$ and $\pi(r_\alpha) \force r_{\alpha} \leq \dot{q}_{\alpha}$ in $Q / \dot{G}$. Since $Q$ has the $(\lambda,\lambda,<\mu)$-c.c., there is a $K \in [\lambda]^{\lambda}$ such that $\forall Z \in [K]^{<\mu}(\prod_{\alpha \in Z} q_{\alpha} \not= 0)$. Let $b = ||\{\alpha \in K \mid \pi(q_{\alpha}) \in \dot{G}\}| = \lambda||$. Since $P$ has the $\lambda$-c.c., $b \not=0$. 
Let $\dot{K}$ be a $P$-name for $\{\alpha \in K \mid \pi(q_{\alpha}) \in \dot{G}\}$. 

We claim that $b \leq p$ forces $\forall Z \in [\dot{K}]^{<\lambda}(\prod_{\alpha \in Z} q_{\alpha} \in Q / \dot{G})$. By $(<\mu,\lambda)$-distributivity, if $q \leq b$ forces  $\dot{Z} \in [\dot{K}]^{<\mu}$ for some $\dot{Z}$ then we may assume that $q \force \dot{Z} = Z$ for some $Z \in [K]^{<\mu}$. For each $\alpha \in Z$, we have $q \force \pi(r_{\alpha}) \in \dot{G}$ and thus $q \leq \pi(r_{\alpha})$. Because of $q \leq \prod_{\alpha \in Z} \pi(r_{\alpha}) = \pi(\prod_{\alpha \in Z}r_{\alpha})$, $q$ forces $\prod_{\alpha \in Z}r_{\alpha} \in Q / \dot{G}$. $q$ also forces $\prod_{\alpha \in Z}r_{\alpha} \leq r_{\alpha} \leq \dot{q}_{\alpha}$ for each $\alpha \in Z$. 
\end{proof}

Next, we consider the case of layeredness.
\begin{lem}\label{layeredquotient}
 Suppose that $Q$ is $S$-layered for some stationary subset $S \subseteq \lambda$ and $\pi:Q \to P$ is a $2$-continuous projection. Then $P \force Q / \dot{G}$ is $S$-layered.
\end{lem}
\begin{proof}
 We may assume that $P$ and $Q$ are Boolean algebras. Remark that they need not be complete. Let $\langle \mathcal{B}_\delta \mid \delta < \lambda \rangle$ be a filtration of $Q$ with each $\mathcal{B}_\delta$ is a Boolean subalgebra of $Q$. Because $P$ has the $\lambda$-c.c., $S$ remains stationary in the extension by $P$. It is enough to prove that $\mathcal{B}_\delta \lessdot Q$ implies $P \force \mathcal{B}_\delta / \dot{G} \lessdot Q/\dot{G}$ for each $\delta$.

 Let $D = \{q \in Q \mid \exists b \in \mathcal{B}_\delta(b \geq q$ and $b$ is a reduct of $q)\}$. $D$ is dense in $Q$. For each $q \in Q$, $q$ has a reduct $b \in \mathcal{B}_\delta$. It is easy to see that $b$ is a reduct of $q \cdot b$. Thus, $q \cdot b \in D$ and this extends $q$.

 To show $P \force \mathcal{B}_\delta / \dot{G} \lessdot Q/\dot{G}$, take an arbitrary $p \in P$ and $q \in Q$ with $p \force q \in Q / \dot{G}$. We may assume $q \in D$. Thus, $q$ has a reduct $b \geq q$. Because of $p \leq \pi(q) \leq \pi(b)$, $p \force b \in \mathcal{B}_\delta / \dot{G}$. It suffices to show that $p \force \forall c \in \mathcal{B}_\delta / \dot{G}(c \leq b \to c\cdot q \in Q / \dot{G})$. For any $p' \leq p$ and $c \leq b$ with $p' \force c \in \mathcal{B}_\delta / \dot{G}$, Since $b$ is a reduct of $q$ and $\pi$ is $2$-continuous, $p' \leq \pi(c)\cdot \pi(q) = \pi(c \cdot q)$. Thus, $p' \force c \cdot q \in Q / \dot{G}$.
\end{proof}
We get an analogous result for centeredness, although we do not use this in the proof of Theorem~\ref{maintheorem}.
\begin{prop}\label{centeredquotient}
 For a projection $\pi:Q \to P$, suppose that $\pi$ is $<\omega$-continuous and $P \force Q$ is $\lambda$-centered. Then $P \force Q / \dot{G}$ is $\lambda$-centered.
\end{prop}
\begin{proof}
 We may assume that $P$ and $Q$ are Boolean algebras. Let $G$ be an arbitrary $(V,P)$-generic filter. We discuss in $V[G]$. Let $\langle F_{\xi} \mid \xi < \lambda \rangle$ be a centering of $Q$. It is enough to prove that $p_{0},...,p_{n-1} \in F_{\xi}$ implies $p_0 \cdot p_1  \cdots  p_{n-1} \in Q / G \ast H$ for every $p_i \in Q/ G \ast H$. 

 Note that $\pi(p_i) \in G$ for each $i$. Since $F_{\xi}$ is a centered subset, $p_0 \cdot p_1 \cdots p_{n-1} \not= 0$ in $Q$. The $<\omega$-continuity implies $\pi(p_0 \cdot p_1 \cdots p_{n-1})= \pi(p_0) \cdot \pi(p_1) \cdots \pi(p_{n-1}) \in G$, as desired.
\end{proof}

The rest of this section is not used in the proof of Theorem \ref{maintheorem}, but is of independent interest. Refining Lemma \ref{ccquotient}, we characterize the $(\lambda,\lambda,<\mu)$-c.c. of the quotient forcing in terms of the following notion.
\begin{defi}
 For a projection $\pi:Q \to P$ between complete Boolean algebras, we say that $\pi$ is $(\lambda,\lambda,<\mu)$-nice if, for every $X \in [Q]^{\lambda}$, there is a $Y \in [Q]^{\lambda}$ with the following properties:
\begin{itemize}
 \item There is an injection $f:Y \to X$ such that $y \leq f(y)$ for all $y \in Y$.
 \item $\prod Z \not= 0$ and $\pi(\prod Z) = \prod \pi``Z$ for all $Z \in [Y]^{<\mu}$.
\end{itemize} 
\end{defi}

\begin{thm}\label{characterization}
 Suppose that $P$ is $(<\mu,\lambda)$-distributive, $\pi:Q \to P$ is a projection between complete Boolean algebras, and $Q$ has the $(\lambda,\lambda,<\mu)$-c.c. Then the following are equivalent.
\begin{enumerate}
 \item $\pi$ is $(\lambda,\lambda,<\mu)$-nice.
 \item $P \force Q/ \dot{G}$ has the $(\lambda,\lambda,<\mu)$-c.c.
\end{enumerate}
\end{thm}
\begin{proof}
The forward direction can be shown as in the proof of Lemma \ref{ccquotient}. We should check the inverse direction. Let $\{q_{\alpha}\mid \alpha < \lambda\} \subseteq Q$ be arbitrary. We let $b = || |\{\alpha < \lambda \mid \pi(q_{\alpha}) \in \dot{G}\}| = \lambda ||$ and $\dot{K}$ be a $P$-name for $\{\alpha < \lambda \mid \pi(q_{\alpha}) \in \dot{G}\}$. Since $P$ has the $\lambda$-c.c., $b \not= 0$. 
 By the definition of quotient forcing, $b \force \{q_{\alpha} \mid \alpha \in \dot{K}\} \subseteq Q/\dot{G}$. Because $P$ forces that $Q/\dot{G}$ has the $(\lambda,\lambda,<\mu)$-c.c., we can choose $\dot{K}'$ such that $b \force \dot{K}' \in [\dot{K}]^{\lambda}$ and $\prod_{\alpha \in Z}q_\alpha\not= 0$ for all $Z \in [\dot{K}']^{<\mu}$. 

 By the $\lambda$-c.c. of $P$, $K = \{\alpha< \lambda \mid b \cdot ||\alpha \in \dot{K}'||\not=0\}$ is of size $\lambda$. Define $p_{\alpha} = b \cdot ||\alpha \in \dot{K}'||$ for each $\alpha \in K$. There is a $K' \in [K]^{\lambda}$ with $\forall Z \in [K']^{<\omega} (\prod_{\alpha \in K'} p_{\alpha} \not = 0)$. 
 Observe that for every $Z \in [K']^{<\mu}$, $\prod_{\alpha \in Z} p_{\alpha}$ forces $\prod_{\alpha \in Z}q_{\alpha} \in Q / \dot{G}$, and thus, $\prod_{\alpha \in Z}p_\alpha = \prod_{\alpha \in Z} p_{\alpha} \cdot \pi(\prod_{\alpha \in Z}q_{\alpha})$.

 Let $r_{\alpha} = q_{\alpha} \cdot e(p_{\alpha})$, where $e$ is a complete embedding induced by $\pi$. We claim that $\prod_{\alpha \in Z}\pi(q_{\alpha}) = \pi(\prod_{\alpha \in Z} q_{\alpha})$ for every $Z \in [K']^{<\mu}$. This follows by:
 \begin{align*}
  \textstyle{\prod_{\alpha \in Z} \pi(r_{\alpha})} &= \textstyle{\prod_{\alpha \in Z} p_{\alpha}} \\
&= \textstyle{\prod_{\alpha \in Z}p_{\alpha} \cdot \pi(\prod_{\alpha \in Z}q_{\alpha})} \\ &= \textstyle{\pi(\prod_{\alpha \in Z}q_{\alpha} \cdot \prod_{\alpha \in Z}e(p_{\alpha}))} \\
&= \textstyle{\pi(\prod_{\alpha \in Z}q_{\alpha} \cdot e(p_{\alpha}))} = \textstyle{\pi(\prod_{\alpha \in Z}r_\alpha)}.
 \end{align*}
 Thus, $\{r_{\alpha} \mid \alpha \in K'\}$ witnesses to $(\lambda,\lambda,<\mu)$-nice.
\end{proof}

In particular, Knasterness of the quotient forcing can be characterized in term of projections as follows.
\begin{coro}\label{knasterchar}
 Suppose that $\pi:Q \to P$ is a projection between complete Boolean algebras and $Q$ is $\lambda$-Knaster. Then the following are equivalent.
\begin{enumerate}
 \item $\pi$ is $(\lambda,\lambda,2)$-nice.
 \item $P \force Q / \dot{G}$ is $\lambda$-Knaster.
\end{enumerate}
\end{coro}
We will show that Corollary \ref{knasterchar} is not meaningless, that is, (2) does not hold unconditionally. To see this, we use Todor\v{c}evi\'c's construction of a Suslin tree from a Cohen real.
\begin{lem}[Todor\v{c}evi\'c]
 There is an $\langle{e_\alpha:\alpha \to \omega \mid \alpha < \omega_1}\rangle$ with the following properties:
\begin{enumerate}
 \item $\{\xi < \alpha\mid e_\alpha(\xi) \not= e_\beta(\xi)\}$ is finite for all $\alpha < \beta$.
 \item $\{\xi < \alpha \mid e_\alpha(\xi) \leq n\}$ is finite for all $n < \omega$.
\end{enumerate}
\end{lem}

\begin{prop}
 There is a projection $\pi:Q \to P$ between $\aleph_1$-Knaster posets such that $P \force Q / \dot{G}$ is not $\aleph_1$-Knaster. In particular, $\pi$ is not $(\aleph_1,\aleph_1,2)$-nice.
\end{prop}
\begin{proof}

Let $C$ be a Cohen forcing, that is, $C = {^{<\omega} \omega}$. Let $\dot{c}$ be a $C$-name such that $C \force \dot{c} = \bigcup \dot{G}$. Todor\v{c}evi\'c showed that $C$ forces that the poset $\dot{T} = \{\dot{c} \circ e_\alpha \upharpoonright \beta \mid \beta \leq \alpha <\omega_1\}$, ordered by reverse inclusion, has the $\aleph_1$-c.c. and is not $\aleph_1$-Knaster. We refer to~\cite{MR2355670} for more details.

 Let $P = C$, $Q = C \ast \dot{T}$ and $\pi:Q \to P$ be a natural projection. Of course, $P \force Q / \dot{G} \simeq \dot{T}$ is not $\aleph_1$-Knaster. 

 It remains to show that $Q$ is $\aleph_1$-Knaster. Let $X = \{ \langle p_i,\dot{c} \circ e_{\alpha_i} \upharpoonright \beta_i\rangle \mid i < \omega_1 \}$ be arbitrary. Shrinking $X$, there are $K \in [\omega_1]^{\omega_1}$ and $p$ such that $p_i = p$ for all $i\in K$. For each $i \in K$, $a_i = \{\xi < \alpha_i \mid e_{\alpha_i}(\xi) \leq |p|\}$ is finite. The usual $\Delta$-system argument takes $K' \in [\omega_1]^{\omega_1}$ and $r$ such that $a_i \cap a_j = r$ for each $i<j$ in $K'$. Note that the number of functions that has a form of $e_\alpha \upharpoonright r$ is $\omega$ at most. There is a $K'' \in [K']^{\omega_1}$ such that $e_{\alpha_i} \upharpoonright a = e_{\alpha_j} \upharpoonright a$ for each $i < j$ in $K''$. We claim that any two elements in $Y = \{\langle p_i,\dot{c} \circ e_{\alpha_i} \upharpoonright \beta_i\rangle \mid i \in K''\}$ are compatible.

 Fix a pair $i < j$ in $K''$. For every $\xi$, if $e_{\alpha_i}(\xi), e_{\alpha_j}(\xi) < |p|$ then $\xi \in r$, which in turn implies $e_{\alpha_i}(\xi) = e_{\alpha_j}(\xi)$. This ensures us, for every $\xi$ with $e_{\alpha_i}(\xi) \not= e_{\alpha_j}(\xi)$, one of the following holds:
\begin{itemize}
 \item $e_{\alpha_i}(\xi),e_{\alpha_j}(\xi) \geq |p|$.
 \item $e_{\alpha_i}(\xi) \geq |p|$ and $e_{\alpha_j}(\xi) < |p|$.
 \item $e_{\alpha_j}(\xi) \geq |p|$ and $e_{\alpha_i}(\xi) < |p|$.
\end{itemize}
 Since $\Delta = \{\xi \mid e_{\alpha_{i}}(\xi) \not= e_{\alpha_{i}}(\xi)\}$ is finite, $m = \max (e_{\alpha_i}``\Delta) \cup (e_{\alpha_{j}}``\Delta) + 1$ is a natural number.
 Define $q \in {^{m}\omega}$ by 
\begin{center}
 $q(n) = \begin{cases}p(n) & n < |p|\\
	    p(e_{\alpha_j}(\xi)) & \text{there is a }\xi\text{ such that }n = e_{\alpha_i}(\xi)\text{ and }e_{\alpha_j}(\xi)<|p|\\
	    p(e_{\alpha_i}(\xi)) & \text{there is a }\xi\text{ such that }n = e_{\alpha_j}(\xi)\text{ and }e_{\alpha_i}(\xi)<|p|\\
	  0 & \text{otherwise}
	   \end{cases}$
\end{center}
 It is easy to see that $\langle q,\dot{c} \circ e_{\alpha_k} \upharpoonright \beta_k\rangle$ is a common extension of $\langle p,\dot{c} \circ e_{\alpha_i} \upharpoonright \beta_i\rangle$ and $\langle p,\dot{c} \circ e_{\alpha_j} \upharpoonright \beta_j\rangle$, here $k$ is $i$ or $j$ such that $\beta_i,\beta_j \leq \beta_k$.
\end{proof}

\section{Diagonal product of Levy collapses}

In this paper, we use a slight modification of Levy collapse. First, we write $[\kappa,\lambda)_{\mu\text{-cl}}$ for the set of all $\mu$-closed cardinal in $[\kappa,\lambda)$. Here, a $\mu$-closed cardinal is a cardinal $\gamma$ with ${\gamma}^{<\mu} = \gamma$. For regular cardinals $\mu < \kappa$, our Levy collapse $\mathrm{Coll}(\mu,<\kappa)$ is the $<\mu$-support product $\prod_{\gamma \in [\mu^{+},\kappa)_{\mu\text{-cl}}}^{<\mu} {^{<\mu}\gamma}$. $\mathrm{Coll}(\mu,<\kappa)$ is $\mu$-directed closed. We remark that our Levy collapse $\mathrm{Coll}(\mu,<\kappa)$ is forcing equivalent to the usual one and $\mathrm{Coll}(\mu,<\kappa)$ has the $\kappa$-c.c. if $\kappa$ is inaccessible.

For regular cardinals $\mu < \kappa$, the diagonal product of Levy collapses is $P(\mu,\kappa) = \prod_{\alpha \in [\mu,\kappa) \cap \mathrm{Reg}}^{<\mu}\mathrm{Coll}(\alpha,<\kappa)$. It is easy to see that $P$ is $<\mu$-directed closed. 
\begin{lem}
 If $\kappa$ is inaccessible, then $P(\mu,\kappa)$ has the $(\kappa,\kappa,<\mu)$-c.c.
\end{lem}
\begin{proof}
 For any $X \in [P(\mu,\kappa)]^{\kappa}$, the usual $\Delta$-system argument takes $Y \in [X]^{\kappa}$ and $r$ with the following properties:
\begin{itemize}
 \item $\{\mathrm{supp}(p) \mid p \in Y\}$ is a $\Delta$-system with root $r$.
 \item $r \subseteq \kappa$ is bounded by some regular cardinal $\eta < \kappa$. 
\end{itemize}
 For each $\alpha \in r$, we can see that $p(\alpha)$ is a partial function from  $\alpha \times [\alpha,\kappa)_{\alpha^{+}\text{-cl}}$ to $\kappa$. Note that $|\bigcup_{\alpha \in r}\{\alpha\} \times \mathrm{dom}(p(\alpha))| < \eta$. Again, the usual $\Delta$-system argument takes $Y' \in [Y]^{\kappa}$, $r'$ and $q$ such that
\begin{itemize}
 \item $\{\{\alpha\} \times \mathrm{dom}(p(\alpha)) \mid \alpha \in Y'\}$ is a $\Delta$-system with root $r'$.
 \item $q \in P(\mu,\kappa)$. 
 \item For all $p \in Y'$ and $\alpha \in r$, $p(\alpha) \upharpoonright \{\xi \mid \langle{\alpha,\xi}\rangle \in r'\} = q(\alpha)$.
\end{itemize}
It is easy to see that $Y'$ works.
\end{proof}
From this, $P(\mu,\kappa)$ forces $\mu^{+} = \kappa$ if $\kappa$ is inaccessible. 
\begin{lem}\label{diagonallayered}
Suppose $\kappa$ is inaccessible.
\begin{enumerate}
 \item If $\kappa$ is Mahlo, then $P(\mu,\kappa)$ is $[\mu,\kappa) \cap \mathrm{Reg}$-layered. 
 \item If $\kappa$ is not Mahlo, then $P(\mu,\kappa)$ is not $S$-layered for all stationary subsets $S \subseteq \kappa$.
\end{enumerate}
\end{lem}
To prove Lemma \ref{diagonallayered}, we need the following lemma.
\begin{lem}\label{regularchar}Suppose $\kappa$ is inaccessible.
 \begin{enumerate}
  \item $P(\mu,\delta) \lessdot P(\mu,\kappa)$ for all $\delta < \kappa$.
  \item There is a club $C$ such that $\bigcup_{\eta < \delta} P(\mu,\eta) \lessdot P(\mu,\kappa)$ if and only if $\delta$ is regular for all $\delta \in C$.
 \end{enumerate}
\end{lem}
\begin{proof}
(1) is easy. Let us see $(2)$. It is easy to see $P(\mu,\delta) \supseteq \bigcup_{\eta < \delta}P(\mu,\eta)$. 

Let $C = \{\delta < \kappa \mid \forall \eta < \delta(\eta^{<\eta} < \delta)$ and $\delta$ is a limit cardinal$\}$. $C$ is a club in $\kappa$. Note that $\sup [\alpha^{+},\delta)_{\alpha\text{-cl}} = \delta$ for each $\delta \in C$ and $\alpha < \delta$. 

In the case of $\delta \in C$ regular, $P(\mu,\delta) = \bigcup_{\eta < \delta}P(\mu,\eta) \lessdot P(\mu,\kappa)$, by (1). If $\delta \in C$ is singular, there is a regular cardinal $\alpha$ with $\mathrm{cf}(\delta) < \alpha < \delta$. Then $\sup[\alpha^{+},\delta)_{\alpha\text{-cl}} = \delta$. Let $\{\delta_i \mid i < \mathrm{cf}(\delta)\} \subseteq [\alpha^{+},\delta)_{\alpha\text{-cl}}$ be a sequence which converges to $\delta$. Define $p \in P(\mu,\delta)$ by, 
\begin{itemize} 
 \item $\mathrm{supp}(p) = \{\alpha\}$.
 \item $p(\alpha) \in \mathrm{Coll}(\alpha,<\delta)$ is such that 
       \begin{itemize}
	\item $\mathrm{dom}(p(\alpha)) = \{\delta_i \mid i < \mathrm{cf}(\delta)\}$, and
	\item $p(\alpha)(\delta_i) = \begin{cases} \{\langle 0,\delta_{i-1}\rangle\} &i\text{ is successor ordinal} \\ \{\langle 0,0\rangle\} &\text{otherwise}  \end{cases}$
       \end{itemize}
\end{itemize}
It is easy to see $p(\alpha) \in \mathrm{Coll}(\alpha,<\delta) \setminus \bigcup_{\eta < \delta}\mathrm{Coll}(\alpha,<\eta)$. In particular, $p$ does not have a reduct in $\bigcup_{\eta < \delta} P(\mu,\eta)$.
\end{proof}

\begin{proof}[Proof of Lemma \ref{diagonallayered}]
 Let $C$ be a club in Lemma~\ref{regularchar}. For (1), by Lemma \ref{regularchar}, $P(\mu,\kappa)$ is $[\mu,\kappa)\cap \mathrm{Reg}$-layered witnessed by $\langle P(\mu,\delta) \mid \delta < \kappa \rangle$. 

 For (2), by the assumption, there is a club $D \subseteq C$ such that every element in $D$ are singular. Define $Q_\delta = \bigcup_{\eta<\delta}P(\mu,\eta)$. $\langle Q_{\delta} \mid \delta < \kappa\rangle$ is a filtration of $P(\mu,\kappa)$. By Lemma \ref{singularcase}, $Q_\delta \not\mathrel{\lessdot} P(\mu,\kappa)$ for all $\delta \in D$. By Lemma \ref{layeredchar}, $P(\mu,\kappa)$ is not $S$-layered for all stationary subsets $S \subseteq \kappa$. 
\end{proof}

The following lemma is contained in the proof of Lemma \ref{regularchar} (2), and is used in the proof of Claim \ref{layeredconclusion}.
\begin{lem}\label{singularcase}
 For inaccessible $\kappa$, let $C$ be a club in Lemma \ref{regularchar}.(2). For every singular $\delta \in C$, there is a $p \in P(\mu,\delta) \setminus \bigcup_{\eta < \delta} P(\mu,\eta)$ with the following properties:
\begin{itemize}
 \item $\mathrm{supp}(p) \cap (\lambda + 1) = \emptyset$, and, 
 \item For every $q \in \bigcup_{\eta < \delta} P(\mu,\eta)$, there is an $r \in \bigcup_{\eta < \delta} P(\mu,\eta)$ such that $\mathrm{dom}(r) \cap (\lambda + 1) = \emptyset$, $r \perp p$ in $P(\mu,\delta)$ and $r \cdot q \in \bigcup_{\eta < \delta} P(\mu,\eta)$.
\end{itemize}
\end{lem}
\begin{proof}
 The condition $p$ which was defined in the proof of Lemma~\ref{regularchar} works.
\end{proof}

The following property of Levy collapses is used in the proof of Claim \ref{centeredconclusion}.
\begin{lem}\label{levycentered2}
 For inaccessible $\lambda$ and regular $\kappa < \alpha < \lambda$, $\mathrm{Coll}(\kappa,<\lambda)$ forces $\mathrm{Coll}^{V}(\alpha,<\lambda)$ is not $\kappa$-centered.
\end{lem}
\begin{proof}
 We show by contradiction. We may assume that $\mathrm{Coll}(\kappa,\lambda)$ forces that $\mathrm{Coll}^{V}(\alpha,<\lambda)$ is $\kappa$-centered. Let $\langle \dot{F}_{\xi} \mid \xi < \kappa \rangle$ be a $\mathrm{Coll}(\kappa,<\lambda)$-name for a centering. We may assume that it is forced that each $\dot{F}_{\xi}$ is a filter because $\prod X \in \mathrm{Coll}(\kappa,<\lambda)$ for every $X \subseteq \mathrm{Coll}(\kappa,<\lambda)$ with $X$ has a lower bound.

For each $\xi < \kappa$ and $q \in \mathrm{Coll}(\alpha,<\lambda)$, let $\rho(q,\xi)$ be the least cardinal $\delta$ such that there is a maximal anti-chain $\mathcal{A} \subseteq \mathrm{Coll}(\kappa,<\delta)$ with $\forall p \in \mathcal{A}(p$ decides $q \in \dot{F}_{\xi})$. Let $D \subseteq \lambda$ be a club generated by a mapping $\delta \mapsto \sup\{\rho(q,\xi) \mid \xi < \kappa \land q \in {\mathrm{Coll}(\alpha,<\delta)}\}\cup\{\delta^{<\delta}\}$. 

Fix a $\delta \in D \cap E^{\lambda}_{\geq \kappa} \cap E^{\lambda}_{<\alpha}$. The following holds now.
 \begin{itemize}
  \item $|\mathrm{Coll}(\kappa,<\delta)| = \delta$, in particular, $\mathrm{Coll}(\kappa,<\delta) \force (\delta^{+})^V \geq \kappa^{+}$. 
  \item $\mathrm{Coll}(\alpha,<\delta)$ has an anti-chain of size $\delta^{\mathrm{cf}(\delta)} \geq \delta^{+}$. 
 \end{itemize}
 The first item follows by the standard cardinal arithmetic. Let us define an anti-chain for $\mathrm{Coll}(\alpha,<\delta)$ of size $\delta^{\mathrm{cf}(\delta)}$. Note that we can choose a sequence $\langle\delta_{i} \mid i < \mathrm{cf}(\delta)\rangle \subseteq [\alpha^{+},\delta)_{\alpha\text{-cl}}$ which converges to $\delta$. For each $i < \mathrm{cf}(\delta)$, ${^{<\alpha}\alpha_{i}}$ has an anti-chain $\{p^{i}_{\xi} \mid \xi < \alpha_{i}\}$ of size $\alpha_i$. For each $f \in \prod_{i < \mathrm{cf}(\delta)}\alpha_{i}$, define $p_f \in \mathrm{Coll}(\alpha,<\lambda)$ as follows:
\begin{itemize}
 \item $\mathrm{supp}(p_f) = \{\alpha_i \mid i < \mathrm{cf}(\delta)\}$. 
 \item $p_f(\alpha_i) = p_{f(i)}^i$. 
\end{itemize}
It is easy to see that $\{p_f \mid f \in  \prod_{i < \mathrm{cf}(\delta)}\alpha_{i}\}$ witnesses.

Let $G$ be an arbitrary $(V,\mathrm{Coll}(\kappa,<\lambda))$-generic. $G$ can be factored as $G = G_0 \times G_1$ where $G_0$ is a $(V,\mathrm{Coll}(\kappa,<\delta))$-generic. Let us discuss in $V[G_0]$. Letting $Q = (\bigcup_{\zeta < \delta}\mathrm{Coll}(\alpha,<\zeta))^{V}$. Let $F_\xi = \dot{F}_{\xi}^{G}$ in $V[G]$, note that $F_\xi \cap Q \in V[G_0]$ by $\delta \in D$. In particular, $Q$ has a centering $\langle {F}_\xi \cap Q \mid \xi < \kappa \rangle$ in $V[G_0]$. Define $H_\xi = \{q \in \mathrm{Coll}^V(\alpha,<\delta) \mid \forall \alpha \in \mathrm{supp}(q)(q \upharpoonright (\mathrm{supp}(q) \cap \alpha) \in F_{\xi} \cap Q)\}$. We claim that $\langle H_\xi \mid \xi < \kappa \rangle$ is a centering for $\mathrm{Coll}^V(\alpha,<\delta)$. It is easy to see that each $H_\xi$ is a filter. For each $q \in \mathrm{Coll}^V(\alpha,<\delta)$, in $V[G]$, there is a $\xi$ such that $q \in F_\xi$. For every $\alpha < \mathrm{supp}(q)$, $q \upharpoonright (\mathrm{supp}(q) \cap \alpha) \in Q \cap F_{\xi}$. This has held in $V[G_0]$ yet, and thus, $q \in H_{\xi}$ in $V[G_0]$. 

 We showed that $\mathrm{Coll}^V(\alpha,<\delta)$ is $\kappa$-centered, which in turn implies the $\kappa^{+}$-c.c. But $\mathrm{Coll}^V(\alpha,<\delta)$ has a maximal anti-chain of size $(\delta^{\mathrm{cf}(\delta)})^V \geq \kappa^{+}$ as we have seen. This is a contradiction.
\end{proof}

\begin{rema}\label{remarklevy}
 For inaccessible $\lambda$ and regular $\alpha \leq \kappa$, $\mathrm{Coll}(\kappa,<\lambda)$ forces $\mathrm{Coll}^{V}(\alpha,<\lambda)$ is $\kappa$-centered.
\end{rema}
\begin{proof}
 We discuss in the extension by $\mathrm{Coll}(\kappa,<\lambda)$. For all $\gamma \in [\alpha^{+},\lambda)_{\alpha\text{-cl}}$, because of $|({^{<\alpha}\gamma})^V| \leq \kappa$, $({^{<\alpha}\gamma})^{V}$ is $\kappa$-centered. By Lemma 4 in~\cite{MR925267}, it follows that $\prod_{\gamma \in [\alpha^{+},\lambda)}^{<\alpha}({^{<\alpha}\gamma})^{V}$ is $\kappa$-centered. In particular, $\mathrm{Coll}^{V}(\alpha,<j(\kappa))$ is $\kappa$-centered. 
\end{proof}
By Lemma \ref{levycentered2} and Remark \ref{remarklevy}, we have there is no complete embedding from $\mathrm{Coll}(\alpha,<\lambda)$ to $\mathrm{Coll}(\kappa,<\lambda)$ for all $\alpha \in [\kappa,\lambda)\cap \mathrm{Reg}$. On the other hand, the diagonal product of Levy collapses have a nice absorption properties as follows. Lemma \ref{mainprojection} plays an important role in the proof of Theorem \ref{maintheorem}.
\begin{lem}\label{mainprojection}
 Suppose that $\mu < \kappa \leq \lambda < \nu$ are regular and $\kappa$ and $\nu$ are inaccessible. Then there is a projection $\pi:P(\mu,\nu) \to P(\mu,\kappa) \ast \dot{\mathrm{Coll}}(\lambda,<\nu)$ which is continuous. In addition, $\pi(p) = \langle p,\emptyset\rangle$ for all $p \in P(\mu,\kappa)$.
\end{lem}

To prove this, we need
\begin{lem}\label{levyprojection}
 Suppose that $P$ has the $\kappa$-c.c. and $|P| \leq \kappa$. For inaccessible $\lambda > \kappa$, there is a projection $P \times \mathrm{Coll}(\kappa,<\lambda) \to P \ast \dot{\mathrm{Coll}}(\kappa,<\lambda)$ which is identity on the first coordinate and continuous.
\end{lem}

In the proof of Lemma \ref{levyprojection}, we use knowledge of the term forcing. For a poset $P$ and a $P$-name $\dot{Q}$ for a poset, the term forcing $T(P,\dot{Q})$ is a complete set of representatives from $\{\dot{q} \mid \force \dot{q} \in \dot{Q}\}$ with respect to the canonical equivalence relation. $T(P,\dot{Q})$ is ordered by $\dot{q} \leq \dot{q}'\leftrightarrow \force \dot{q} \leq \dot{q}'$. The following lemma is due to Laver. 
\begin{lem}[Laver]\label{termforcing}
 ${\mathrm{id}}:P \times T(P,\dot{Q}) \to P \ast \dot{Q}$ is a projection.
\end{lem}
 The proof of the following lemma is based on Shioya~\cite{MR4159767}. 
\begin{lem}\label{termproduct}
 Suppose $P$ has the $\kappa$-c.c. and $|P| \leq \kappa$. Then the following holds:
\begin{enumerate}
 \item If $\gamma$ is $\kappa$-closed, then there is a dense embedding from ${^{<\kappa}\gamma}$ onto $T(P,\dot{{^{<\kappa}\gamma}})$.
 \item If $\langle\dot{Q}_{\gamma} \mid \gamma \in I\rangle$ is a sequence of $P$-names of a poset. Then there is a dense embedding from $\prod_{\gamma \in I}^{<\kappa} T(P,\dot{Q}_{\gamma})$ onto $T(P,\prod_{\gamma \in I}^{<\kappa}\dot{Q}_{\gamma})$
\end{enumerate}
\end{lem}
\begin{proof}
 (1) Note that $D = \{\dot{q} \in T(P,\dot{{^{<\kappa}}\gamma}) \mid \exists \delta(\force \mathrm{dom}(\dot{q}) = \delta)\}$ is dense in $T(P,\dot{{^{<\kappa}}\gamma})$. For each $\dot{p} \in T(P,\dot{{^{<\kappa}}\gamma})$, by the $\kappa$-c.c. of $P$, there is a $\delta < \kappa$ with $\force \mathrm{dom}(\dot{p}) < \delta$. The usual density argument takes $\dot{q} \in D$ with $\force \dot{q} \leq \dot{p}$.

By the assumption, there is a sequence $\langle \dot{\tau}_\alpha \mid \alpha < \gamma \rangle$ of $P$-names for ordinals below $\gamma$ with the following properties:
\begin{itemize}
 \item $\force \dot{\tau} \in \gamma$ implies $\force \dot{\tau} = \dot{\tau}_\alpha$ for some $\alpha$.
 \item $\not\force \dot{\tau}_\alpha = \dot{\tau}_\beta$ for all $\alpha < \beta$.
\end{itemize}
 For each $p \in {^{<\kappa}\gamma}$, the mapping which sends $p$ onto $\langle{\dot{\tau}_{p(\xi)} \mid \xi \in \mathrm{dom}(p)}\rangle$ is an isomorphism between ${^{<\kappa}\gamma}$ and $D$. This is a required embedding.

(2) Note that $E = \{\dot{q} \in T(P,\prod_{\gamma \in I}^{<\kappa}\dot{Q}_{\gamma}) \mid \exists d \subseteq I( \force \mathrm{supp}(\dot{q}) = d)\}$ is dense in $T(P,\prod_{\gamma \in I}^{<\kappa}\dot{Q}_{\gamma})$, as we have seen in (1). The natural isomorphism from $\prod_{\gamma \in I}^{<\kappa} T(P,\dot{Q}_{\gamma})$ onto $E$ works. 
\end{proof}
From this, we have
\begin{proof}[Proof of Lemma~\ref{levyprojection}]
 We remark that $P$ does not change the class of all $\kappa$-closed cardinals. The required projection follows by,
 \begin{align*}
  P \times \mathrm{Coll}(\kappa,<\lambda) &= P \times \textstyle{\prod_{\gamma \in [\kappa^{+},<\lambda)_{\kappa\text{-cl}}}^{<\kappa} {^{<\kappa}\gamma}} \\
  &\to P \times \textstyle{\prod_{\gamma \in [\kappa^{+},<\lambda)_{\kappa\text{-cl}}}^{<\kappa} T(P,\dot{{^{<\kappa}\gamma}})} \\ 
  &\to P \times T(P,\textstyle{\prod_{\gamma \in [\kappa^{+},<\lambda)_{\kappa\text{-cl}}}^{<\kappa} \dot{{^{<\kappa}\gamma}}}) \\ 
  &= P \times T(P,\dot{\mathrm{Coll}}(\kappa,<\lambda)) \to P \ast \dot{\mathrm{Coll}}(\kappa,<\lambda).
 \end{align*}
 The second line and the third line follow by Lemma~\ref{termproduct} (1) and (2), respectively. The last line follows from Lemma~\ref{termforcing}. 

For continuity, we may assume that $P$ is a complete Boolean algebra. For $q\in \mathrm{Coll}(\kappa,<\lambda)$, let $\dot{q}$ be a $P$-name such that $\pi_0(\emptyset,q) = \langle \emptyset, \dot{q} \rangle$. Then $\dot{q}$ is a $P$-name such that
\begin{itemize}
 \item $P \force \mathrm{dom}(\dot{q}) = \mathrm{dom}(q) \subseteq \lambda \times [\kappa^{+},\lambda)_{\kappa\text{-cl}}$.
 \item $P \force \dot{p}(\xi,\zeta) = \dot{\tau}_{p(\xi,\zeta)}^{\zeta}$ for all $\langle \xi,\zeta \rangle \in \mathrm{dom}(q))$.
\end{itemize}
Here, $\langle{\dot{\tau}^{\zeta}_\alpha \mid \alpha < \zeta}\rangle$ is a sequence of $P$-names defined in Lemma~\ref{termproduct} and we identify an element of $\mathrm{Coll}(\kappa,<\lambda)$ with a partial function from $\kappa \times [\kappa^{+},\lambda)_{\kappa\text{-cl}}$ to $\lambda$. We have $\pi_0(p,q) = \langle p,\dot{q} \rangle$ for all $\langle p,q \rangle \in P \times \mathrm{Coll}(\kappa,<\lambda)$.

 Consider a set $Z = \{ \langle p_i,q_i\rangle \mid i < \nu \} \subseteq P \times \mathrm{Coll}(\kappa,<\lambda)$ such that $\prod Z \in P \times \mathrm{Coll}(\kappa,<\lambda)$. We remark that $\prod Z = \langle \prod_i p_i,\bigcup_i q_i \rangle$. We let $p = \prod_i p_i$ and $q = \bigcup_i q_i$. Our goal is showing $\pi_0(p,q) = \prod \pi_0``Z$. Note that
\begin{itemize}
 \item $P \force \dot{q}_i(\xi,\zeta) = \dot{\tau}^{\zeta}_{q_i(\xi,\zeta)} = \dot{\tau}^{\zeta}_{q(\xi,\zeta)} = \dot{q}(\xi,\zeta)$ for each $\langle\xi,\zeta\rangle \in \mathrm{dom}(p_i)$. 
 \item $P \force \mathrm{dom}(\dot{q}) = \mathrm{dom}(q) = \bigcup_{i < \nu}\mathrm{dom}({p}_i) = \bigcup_{i < \nu}\mathrm{dom}(\dot{p}_i)$.
\end{itemize}
 Therefore, $p \force \dot{q} = \bigcup_{i}\dot{p}_i = \prod_i \dot{p}_i$. In particular, we have $\pi_0(p,q) = \langle p,\dot{q} \rangle = \prod \pi_0``Z$.
\end{proof}

\begin{proof}[Proof of Lemma \ref{mainprojection}]
$\pi$ is defined as 
\begin{align*}
 P(\mu,\nu) & = \textstyle{\prod^{<\mu}_{\alpha \in [\mu,\nu)\cap \mathrm{Reg}} \mathrm{Coll}(\alpha,<\nu)} \\ &\to \textstyle{\prod^{<\mu}_{\alpha \in [\mu,\kappa)\cap \mathrm{Reg}} \mathrm{Coll}(\alpha,<\nu) \times \mathrm{Coll}(\lambda,<\nu)} \\ 
&\to \textstyle{\prod^{<\mu}_{\alpha \in [\mu,\kappa)\cap \mathrm{Reg}} \mathrm{Coll}(\alpha,<\kappa) \times \mathrm{Coll}(\lambda,<\nu)} \\
&= P(\mu,\kappa) \times \mathrm{Coll}(\lambda,<\nu) \\
&\to P(\mu,\kappa) \ast \dot{\mathrm{Coll}}(\lambda,<\nu).
\end{align*} 
The last line follows by Lemma \ref{levyprojection}. By the definition of $\pi$, $\pi(p) = \langle p,\emptyset \rangle$ for all $p \in P(\mu,\kappa)$.
\end{proof}

\section{Proof of Theorem \ref{maintheorem}}
In this section is devoted to 
\begin{proof}[Proof of Theorem \ref{maintheorem}]

We let $P = P(\mu,\kappa)$. We remark that $j(P) = P(\mu,j(\kappa))$ has the $j(\kappa)$-c.c. by the almost-hugeness. By Lemma \ref{mainprojection}, we have a continuous projection from $\pi:j(P) \to P \ast \dot{\mathrm{Coll}}(\lambda,<j(\kappa))$. 

Note that $P \subseteq V_{\kappa}$, $P \lessdot j(P)$, and, $\pi$, $\pi(p) = \langle p,\emptyset \rangle$ for all $p \in P$.

Let $G \ast H$ be an arbitrary $(V,P \ast \dot{\mathrm{Coll}}(\lambda,<j(\kappa)))$-generic filter. First, we give a saturated ideal on $\mathcal{P}_{\kappa}\lambda$ in $V[G][H]$. Let $\overline{G}$ be an arbitrary $(V,j(P))$-generic with $\pi``\overline{G} \subseteq G \ast H$. Note that $j``G = G \subseteq \overline{G}$, which in turn implies that $j$ lifts to $j:V[G] \to M[\overline{G}]$ in $V[\overline{G}]$ such that $j(G) = \overline{G}$. By the $j(\kappa)$-c.c. of $j(P)$, ${^{<j(\kappa)}M[\overline{G}]} \subseteq M[\overline{G}]$. Let $m_{\alpha}$ be the coordinate-wise union of $j``(H \cap \mathrm{Coll}^{M[\overline{G}]}(j(\lambda),<\alpha))$. $m_{\alpha} \in \mathrm{Coll}^{M[\overline{G}]}(j(\lambda),<jj(\kappa))$ for all $\alpha < j(\kappa)$ by the closure property of $M[G]$ and the directed closedness of $\mathrm{Coll}^{M[\overline{G}]}(j(\lambda),<jj(\kappa))$. By the $j(\kappa)$-c.c., we can choose a list $\langle X_{\alpha} \mid \alpha < j(\kappa) \rangle$ of $P \ast \dot{\mathrm{Coll}}(\lambda,<j(\kappa))$-names of all subset in $\mathcal{P}_\kappa\lambda$. There is a descending chain $\langle s_{\alpha} \mid \alpha < j(\kappa) \rangle$ with the following properties:
\begin{itemize}
 \item $s_{\alpha}\leq m_{\alpha}$. 
 \item $s_{\alpha}$ decides $j``\lambda \in j(\dot{X}_{\alpha})$.
\end{itemize}
 Let $U = \{\dot{X}^{G \ast H} \mid \exists \beta(s_\beta \force j``\lambda \in j(\dot{X}))\}$. $U$ is a $V[G][H]$-normal $V[G][H]$-ultrafilter over $\mathcal{P}_\kappa\lambda$. Because $\overline{G}$ was an arbitrary $(V,j(P))$-generic with $\pi``\overline{G} \subseteq G \ast H $, we can take a $j(P) / G \ast H$-name $\dot{U}$ for such ultrafilter. Let $I$ be define by 
 \begin{center}
  $X \in I$ if and only if $j(P) / G \ast H \force \mathcal{P}_{\kappa} \lambda \setminus X \in \dot{U}$. 
 \end{center}
The standard argument shows that $I$ is a normal and fine ideal over $\mathcal{P}_{\kappa}\lambda$. Towards a showing $j(\kappa)$-saturation of $I$, let $\langle{X_\xi \mid \xi \in K}\rangle$ be an anti-chain in $\mathcal{P}(\mathcal{P}_{\kappa}\lambda)$, we have the following:
\begin{itemize}
 \item $||X_\xi \in \dot{U}||\cdot ||X_\zeta \in \dot{U}|| = ||X_\xi \cap X_\zeta \in \dot{U}|| = 0$ for each $\xi \not= \zeta$ in $K$.
 \item $||X_{\xi} \in \dot{U}|| \not= 0$ for each $\xi \in K$.
\end{itemize}
 It follows that $\{||X_{\xi} \in \dot{U}|| \mid \xi \in K\}$ is an anti-chain in $\mathcal{B}(j(P) / G \ast H)$. Note that $j(P) = P(\mu,j(\kappa))$ has the $j(\kappa)$-c.c., and thus $j(P) / G \ast H$ has the $j(\kappa)$-c.c. Therefore $|K| < j(\kappa)$, as desired.
 
\begin{clam}\label{duality}
 $j(P)/G \ast H \simeq \mathcal{P}(\mathcal{P}_\kappa\lambda) / I$.
\end{clam}
\begin{proof}[Proof of Claim]
 
The proof is based on Foreman--Magidor--Shelah~\cite{MR942519}. As in the previous argument, let us consider a mapping $\tau:\mathcal{P}(\mathcal{P}_\kappa\lambda) / I \to \mathcal{B}(j(P)/G \ast H)$ that sends $X$ to $||X \in \dot{U}||$. The standard argument shows that $\tau$ is a complete embeding and $\dot{U}$ is a $j(P)/G \ast H$-name for $(V[G][H],\mathcal{P}(\mathcal{P}_\kappa\lambda) / I)$-generic filter generated by $\tau^{-1}\dot{\overline{G}}$. Here, $\dot{\overline{G}}$ is the canonical name of $(V[G][H],j(P)/G \ast H)$-generic filter. It is enough to prove that $\tau$ is a dense embedding. 

We claim that there is an $f_q$ such that $||\{a \in \mathcal{P}_\kappa\lambda \mid f_q(a) \in G\} \in \dot{U}|| = q$ for every $q \in j(P) / G \ast H$. It follows that the range of $\tau$ is a dense subset in $\mathcal{B}(j(P) / G \ast H)$. 

Let $\overline{G}$ be an arbitrary $(V,j(P))$-generic filter with $\pi``\overline{G} \subseteq G \ast H$. Note that $q \in j(P) \cap V_{\beta}$ for some $\beta < j(\kappa)$. By the elementarity of $j$ and $j$ is almost-huge, we can choose inaccessible $\alpha < j(\kappa)$ with $\alpha > \beta$. Let $U_\alpha = \dot{U}^{\overline{G}} \cap \mathcal{P}(\mathcal{P}_\kappa\lambda)^{V[G][H\upharpoonright\alpha]}$. By the definition of $\dot{U}$, we can choose a $(V[\overline{G}],\mathrm{Coll}^{M[\overline{G}]}(j(\lambda),<jj(\kappa)))$-generic filter $\overline{H}$ such that, 
\begin{itemize}
 \item $j$ lifts to $j:V[G][H\upharpoonright \alpha] \to M[\overline{G}][\overline{H}\upharpoonright {j(\alpha)}]$ and $j(G) = \overline{G}$.
 \item $X \in U_\alpha$ if and only if $j``\lambda \in j(X)$ in $M[\overline{G}][\overline{H}\upharpoonright j(\alpha)]$.
\end{itemize}
Here, $H\upharpoonright \alpha = H \cap \mathrm{Coll}(\lambda,<\alpha)$ and $\overline{H} \upharpoonright j(\alpha) = \overline{H} \cap {\mathrm{Coll}^{M[\overline{G}]}(j(\lambda),<j(\alpha))}$. We can consider the following commutative diagram of elementary embeddings. 
   \begin{center}
   \begin{tikzpicture}[auto,->]
    \node (vtwo) at (0,-1) {$V[G][H\upharpoonright \alpha]$};
    \node (mtwo) at (5,-1) {$M[\overline{G}][\overline{H}\upharpoonright j(\alpha)]$};

    \node (n) at (2.5,-3) {$N$};
    \draw (vtwo) --node {$\scriptstyle {j}$} (mtwo);
    \draw (vtwo) --node[swap] {$\scriptstyle i$} (n);
    \draw (n) --node[swap] {$\scriptstyle k$} (mtwo);
   \end{tikzpicture}
  \end{center}
 Here, $i:V \to N \simeq \mathrm{Ult}(V[G][H\upharpoonright \alpha],U_\alpha)$ is an ultrapower mapping and $k$ is defined by $k([f]_{\dot{U}_\alpha}) = j(f)(j``\lambda)$. It is easy to see that $k$ is an elementary embedding. We claim $\mathrm{crit}(k) \geq \alpha$. Because $\alpha$ is inaccessible, $\lambda^{+} = \alpha$ in $V[G][H\upharpoonright \alpha]$. We remark that $\mathcal{P}(\lambda)^{V[G][H \upharpoonright \alpha]} \subseteq N$ and $i(\kappa) \geq \alpha = \lambda^{+}$.  $\mathcal{P}(\lambda)^{V[G][H \upharpoonright \alpha]} \subseteq N$ follows by, for each $x$, $x$ can be written as $\{\xi \in \lambda \mid i(\xi) \in i``\lambda\cap i(x)\}$. That is, $x$ is definable in $N$ by the normality of $U_\alpha$. By $\kappa = \mu^{+}$ in $V[G][H\upharpoonright \alpha]$, $N \models i(\kappa)$ is the least cardinal greater than $\mu$. $N$ has no cardinals between $\kappa$ and $\alpha$. On the other hand, $\mathrm{crit}(k) \geq \kappa$ must be cardinal in $N$. Therefore $\mathrm{crit}(k) \geq \alpha$.

 Let us find a name of $q$ in $N$. Since $|V_\beta| = \lambda < \alpha$ holds in $V[G][H \upharpoonright \alpha]$, the same thing holds in $N$.  We can enumerate $i(P) \cap V_\beta^{N}$ as $\langle q_\xi \mid \xi < \lambda \rangle$ in $N$. By $\mathrm{crit}(k) \geq \alpha > \beta$, $k$ is the identity mapping on $V_\beta$, and thus, $k(\langle q_\xi \mid \xi < \lambda \rangle) = \langle q_\xi \mid \xi < \lambda \rangle$. By the elementarity of $k$, $q$ appears in this sequence, that is there exists a $\xi$ such that $q_\xi = k(q_\xi) = q \in N$. We can choose $x$ with $[{x}]_{U_{\alpha}} = q$. Since $\dot{U}$ is $(V[G][H],\mathcal{P}(\mathcal{P}_\kappa\lambda)/ I)$-generic, $\overline{G}$ is an arbitrary, and, $I$ is a saturated ideal, there is an $f_q:\mathcal{P}_\kappa\lambda \to V[G][H]$ such that $\mathcal{P}(\mathcal{P}_\kappa\lambda)/I \force q = [f_q]_{\dot{U}_\alpha}$. 

 Therefore $||\{a \in \mathcal{P}_\kappa \lambda \mid f_q(a) \in G\} \in \dot{U}|| = q$ follows by
\begin{align*}
 \{a \in \mathcal{P}_\kappa\lambda \mid f_q(a) \in G\} \in U &\Leftrightarrow[f_q]_{U_\alpha} = q \in i(G)(\text{for some }\alpha)\\ &\Leftrightarrow k(q) = q \in k \circ i(G) = \overline{G}.
\end{align*}
\end{proof}
By virtue of Claim \ref{duality}, items (1)--(4) follow from the corresponding claims for the quotient forcing $j(P)/G \ast H$. From now on, we discuss in $V$.

Items (1) and (2) follow from Claim \ref{ccconclusion}. 
\begin{clam}\label{ccconclusion}\renewcommand{\labelenumi}{(\roman{enumi})}
 $P \ast \dot{\mathrm{Coll}}(\lambda,<j(\kappa))$ forces that 
 \begin{enumerate}
  \item $j(P) / \dot{G} \ast \dot{H}$ has the $(j(\kappa),j(\kappa),<\mu)$-c.c. 
  \item $j(P) / \dot{G} \ast \dot{H}$ does not have the $(j(\kappa),\mu,\mu)$-c.c. 
 \end{enumerate}
\end{clam}
\begin{proof}[Proof of Claim]
 For (i), We recall that $j(P)$ has the $(j(\kappa),j(\kappa),<\mu)$-c.c. and $P \ast \dot{\mathrm{Coll}}(\lambda,<j(\kappa))$ is $\mu$-closed. By Lemma \ref{ccquotient} and the continuity of $\pi$, it is forced by $P \ast \dot{\mathrm{Coll}}(\lambda,<j(\kappa))$ that $j(P) / \dot{G} \ast \dot{H}$ has the $(j(\kappa),j(\kappa),<\mu)$-c.c. 

 We prove (ii) by contradiction. Suppose otherwise. We consider a set $X= \{r_\alpha \mid \alpha \in [\lambda^{+}, j(\kappa))\cap \mathrm{Reg}\} \subseteq j(P)$ with $\mathrm{supp}(r_\alpha) = \{\alpha\}$ for every $\alpha$. By the definition of $\pi$, $\pi(r_\alpha) = \langle{\emptyset,\emptyset}\rangle$. Therefore $P \ast \dot{\mathrm{Coll}}(\lambda,<j(\kappa)) \force r_\alpha \in j(P) / \dot{G} \ast \dot{H}$ for every $\alpha$. By the assumption, there are $\dot{Z}$, $r \in j(P)$ and $\langle p,\dot{q} \rangle$ such that, $\langle p,\dot{q} \rangle$ forces that 
\begin{itemize}
 \item $r$ is a lower bound in $\{r_\alpha \mid \alpha \in \dot{Z}\}$ in $j(P) / \dot{G} \ast \dot{H}$, and
 \item $|\dot{Z}| = \mu$.
\end{itemize} 
Since $|\mathrm{supp}(r)| < \mu$, we can choose $\beta$ and $\langle p',\dot{q}' \rangle \leq \langle p,\dot{q} \rangle$ such that $\langle p',\dot{q}' \rangle \force \beta \in \dot{Z} \setminus \mathrm{supp}(r)$. Clearly, $\langle p',\dot{q}' \rangle$ does not force that $r \leq r_\beta$ in $j(P)/\dot{G} \ast \dot{H}$ but this is a contradiction. 
\end{proof}

Note that $P \ast \dot{\mathrm{Coll}}(\lambda,<j(\kappa)) \force \mathcal{P}(\mathcal{P}_{\kappa}\lambda) / \dot{I}$ is a complete Boolean algebra because $\dot{I}$ is saturated. The poset $P \ast \dot{\mathrm{Coll}}(\lambda,<j(\kappa))$ forces $|\mathcal{P}(\mathcal{P}_{\kappa}\lambda) /\dot{I}| \leq 2^{\lambda^{<\kappa}} = 2^{\lambda} = j(\kappa)$, and thus, $|\mathcal{B}(j(P) / \dot{G} \ast \dot{H})| \leq j(\kappa)$. Thus, the notion of $S$-layering sequence is not meaningless. Item (3) follows from Claim \ref{layeredconclusion}.
\begin{clam}\label{layeredconclusion}\renewcommand{\labelenumi}{(\roman{enumi})}
\begin{enumerate}
 \item If $j(\kappa)$ is Mahlo, then $P \ast \dot{\mathrm{Coll}}(\lambda,<j(\kappa)) \force j(P) / \dot{G} \ast \dot{H}$ is $S$-layered for some stationary subset $S \subseteq E^{j(\kappa)}_\lambda$. 
 \item If $j(\kappa)$ is not Mahlo, then $P \ast \dot{\mathrm{Coll}}(\lambda,<j(\kappa)) \force  j(P) / \dot{G} \ast \dot{H}$ is not $S$-layered for any stationary subset $S\subseteq \lambda$. 
\end{enumerate} 
\end{clam}
\begin{proof}[Proof of Claim]
 First, we show (i). By Lemma \ref{diagonallayered}, $j(P)$ is $[\mu,j(\kappa))\cap \mathrm{Reg}$-layered. Note that $[\lambda,j(\kappa))\cap \mathrm{Reg}$ remains stationary in the extension. We also remark that $P \ast \dot{\mathrm{Coll}}(\lambda,<j(\kappa))$ forces $[\lambda,j(\kappa))\cap \mathrm{Reg}^{V}\subseteq E^{j(\lambda)}_{\lambda}$ since $P \ast \dot{\mathrm{Coll}}(\lambda,<j(\kappa))$ has the form of $(\kappa$-c.c.$)\ast(\lambda$-closed$)$. 
 By Lemma~\ref{layeredquotient} and the continuity of $\pi$, $P \ast \dot{\mathrm{Coll}}(\lambda,<j(\kappa)) \force j(P) / \dot{G} \ast \dot{H}$ is $[\lambda,j(\kappa))\cap \mathrm{Reg}^{V}$-layered. 

 For (ii), we let $Q_\delta = \bigcup_{\eta < \delta}P(\mu,\eta)$ and $C\subseteq j(\kappa)$ be a club in Lemma \ref{regularchar}.  $\langle Q_\delta \mid \delta < j(\kappa) \rangle$ is a filtration and $Q_\delta \not\mathrel{\lessdot} j(P)$ for all singular $\delta \in C$ by Lemma \ref{singularcase}. Since $j(\kappa)$ is not Mahlo, there is a club $D \subseteq C$ such that each element in $D$ is singular. Note that it is forced that $\langle Q_\delta/ \dot{G}\ast \dot{H} \mid \delta < j(\kappa) \rangle$ is a filtration of $j(P) / \dot{G}\ast \dot{H}$. By Lemma \ref{layeredchar}, it is enough to show that $\force Q_\delta / \dot{G}\ast \dot{H} \not\mathrel{\lessdot} j(P) / \dot{G}\ast \dot{H}$ for all $\delta \in D$. Fix $\delta \in D$, Lemma \ref{singularcase} gives a $p \in P(\mu,\delta)$ with certain properties. By $\pi(p) = \langle \emptyset,\emptyset\rangle$, we have $\force p \in P(\mu,\delta) / \dot{G}\ast \dot{H}$. We claim that $\force p \in P(\mu,\delta) / \dot{G}\ast \dot{H}$ has no reduct in $Q_\delta/\dot{G} \ast \dot{H}$. 

 For any $b \in P \ast \dot{\mathrm{Coll}}(\lambda,<j(\kappa))$ and $q \in Q_\delta$ with $b \force q \in Q_\delta / \dot{G}\ast \dot{H}$, there is an $r \in {Q_\delta}$ such that $\pi(r) = \langle\emptyset,\emptyset\rangle$, $r \perp p$ in $P(\mu,\delta)$, and, $r \cdot q \in Q_\delta$. Then, the following hold:
 \begin{itemize}
  \item $b \leq \pi(q) = \pi(q) \cdot \langle\emptyset,\emptyset\rangle = \pi(q) \cdot \pi(r) = \pi(q \cdot r)$. 
  \item $b \force q \cdot r \leq q$ in $Q_\delta / \dot{G}\ast \dot{H}$ but $(q \cdot r) \perp p$ in $P(\mu,\delta)/ \dot{G}\ast \dot{H}$.
 \end{itemize}
 Thus, $P \ast \dot{\mathrm{Coll}}(\lambda,<j(\kappa))$ forces that $q \in Q_{\delta} / \dot{G} \ast \dot{H}$ is not reduct of $p$, as desired.
\end{proof}
Item (4) follows from Claim \ref{centeredconclusion}.
\begin{clam}\label{centeredconclusion}
 $P \ast \dot{\mathrm{Coll}}(\lambda,<j(\kappa)) \force j(P) / \dot{G} \ast \dot{H}$ is not $\lambda$-centered. 
\end{clam}
\begin{proof}[Proof of Claim]
 Define 
\begin{itemize}
 \item $P_{\ast} = \prod_{\alpha \in [\mu,\lambda + 1)\cap \mathrm{Reg}}^{<\mu}\mathrm{Coll}(\alpha,<j(\kappa))$ and
 \item $P^{\ast} = {\prod^{<\mu}_{\alpha \in [\lambda + 1,j(\kappa))\cap \mathrm{Reg}}\mathrm{Coll}(\alpha,<j(\kappa))}$.
\end{itemize}
 We have $j(P) \simeq P_* \times P^*$. By the definition of $\pi$, it follows that $P \ast \dot{\mathrm{Coll}}(\lambda,<j(\kappa))$ forces $j(P)/ \dot{G} \ast \dot{H} \simeq (P_* / \dot{G} \ast \dot{H}) \times P^*$. Thus, it suffices that $P \ast \dot{\mathrm{Coll}}(\lambda,<j(\kappa))$ forces $P^{*}$ is not $\lambda$-centered. 

Fix $\alpha \in [\lambda^{+},j(\kappa)) \cap \mathrm{Reg}$. Lemma \ref{levyprojection} shows that $P$ forces $\mathrm{Coll}^{V}(\alpha,<j(\kappa))$ is projected to $\mathrm{Coll}^{V^{P}}(\alpha,<j(\kappa))$. Since $P^{\ast}$ is projected to $\mathrm{Coll}^{V}(\alpha,<j(\kappa))$, in the extension by $P \ast \dot{\mathrm{Coll}}(\lambda,<j(\kappa))$, if $P^{\ast}$ is $\lambda$-centered then so is $\mathrm{Coll}^{V^P}(\alpha,<j(\kappa))$.

 But, by Lemma \ref{levycentered2}, $P \ast \dot{\mathrm{Coll}}(\lambda,<j(\kappa))$ forces $\mathrm{Coll}^{V^{P}}(\alpha,<j(\kappa))$ is not $\lambda$-centered for all $\alpha > \lambda$. In particular, $P \ast \dot{\mathrm{Coll}}(\lambda,<j(\kappa))$ forces that $P^*$ is not $\lambda$-centered.
\end{proof}
This complete the proof. \end{proof}

Magidor~\cite{MR526312} gave a model with a normal and countably complete saturated ideal over $[\aleph_3]^{\aleph_1}$ using the universal collapse. We get the same result using the diagonal product of Levy collapses. Indeed,
\begin{thm}
 Suppose that $j$ is a huge embedding with critical point $\kappa$, GCH holds, and $\mu < \kappa < \lambda < j(\kappa)$ are regular cardinals. Then $P(\mu,\kappa) \ast \dot{\mathrm{Coll}}(\lambda,<j(\kappa))$ forces that there is a normal and $\kappa$-complete ideal ${I}$ over $[j(\kappa)]^{\kappa}$ with the following properties:
\begin{enumerate}
 \item $\mathcal{P}([j(\kappa)]^{\kappa})/I$ has the $(j(\kappa),j(\kappa),<\mu)$-c.c.
 \item $\mathcal{P}([j(\kappa)]^{\kappa})/I$ does not have the $(j(\kappa),\mu,\mu)$-c.c.
 \item $\mathcal{P}([j(\kappa)]^{\kappa})/I$ is $S$-layered poset for some stationary subset $S \subseteq E^{j(\kappa)}_{\lambda}$. In particular, it has a dense subset of size $j(\kappa)$.
 \item $\mathcal{P}([j(\kappa)]^{\kappa})/I$ is not $\lambda$-centered.
\end{enumerate}
\end{thm}
\begin{proof}
 Let $P = P(\mu,\kappa)$. Let $G \ast H$ be a $(V,P \ast \dot{\mathrm{Coll}}(\lambda,<j(\kappa)))$-generic. By the proof of Theorem \ref{maintheorem}, it is enough to construct an ideal $I$ over $\mathcal{P}([j(\kappa)]^{\kappa}) / I$ with $\mathcal{P}([j(\kappa)]^{\kappa}) / I \simeq j(P) / G \ast H$. 

 Let $\pi:j(P) \to P \ast \dot{\mathrm{Coll}}(\lambda,<j(\kappa))$ be a projection in Lemma \ref{mainprojection} and $\overline{G}$ be an arbitrary $(V,j(P))$-generic with $\pi `` \overline{G} \subseteq G \ast H$. In $V[\overline{G}]$, $j$ lifts to $j:V[G] \to M[\overline{G}]$. By the GCH, we can choose a list $\langle X_{\alpha} \mid \alpha < j(\kappa)^{+} \rangle$ of $P \ast \dot{\mathrm{Coll}}(\lambda,<j(\kappa))$-names of subset of $[j(\kappa)]^{\kappa}$. Note that ${^{j(\kappa)}}M[\overline{G}] \cap V[\overline{G}] \subseteq M[\overline{G}]$. By $\lambda > \kappa$ and the closure property of $M[\overline{G}]$, $\mathrm{Coll}^{M[\overline{G}]}(j(\lambda),<jj(\kappa))$ is $j(\kappa)^{+}$-directed closed in $V[\overline{G}]$. Let $m$ be the coordinate-wise union of $j ``H$. Then we have $m \in \mathrm{Coll}^{M[G]}(j(\lambda),<jj(\kappa))$ by $|j ``H| = j(\kappa)$. We can construct a decending sequence $\langle s_{\alpha} \mid \alpha < j(\kappa)^{+} \rangle$ such that 
\begin{itemize}
 \item $s_{\alpha} \leq m$.
 \item $s_{\alpha}$ decides $j``j(\kappa) \in j(\dot{X}_{\alpha})$. 
\end{itemize}
 Let $U = \{X \subseteq [j(\kappa)]^{\kappa} \mid s_{\alpha} \force j``j(\kappa) \in X_{\alpha}\}$. There is a $j(P)/{G} \ast {H}$-name $\dot{U}$ for ${U}$. In $V[G][H]$, let $I$ be an induced ideal by $\dot{U}$, that is, $X \in I$ if and only if $j(P) / G \ast H \force [j(\kappa)]^{\kappa} \setminus X \in \dot{U}$. $I$ works as witness.
\end{proof}

\begin{rema}
In the proof of Theorem \ref{maintheorem}, we proved that it is forced by $P(\mu,\kappa) \ast \dot{\mathrm{Coll}}(\lambda,<j(\kappa))$ that $P^{\ast}$ is not $\lambda$-centered. On the other hand, $P_{\ast} / \dot{G} \ast \dot{H}$ is $\lambda$-centered. Indeed, in the extension by $P \ast \dot{\mathrm{Coll}}(\lambda,<j(\kappa))$, we have 
\begin{center}
 $P_{\ast} = (\prod^{<\mu}_{\alpha \in [\mu,\lambda +1)\cap \mathrm{Reg}} \mathrm{Coll}(\alpha,<j(\kappa)))^{V} = \prod^{<\mu}_{\alpha \in [\mu,\lambda+1)\cap \mathrm{Reg}^V} \mathrm{Coll}^{V}(\alpha,<j(\kappa))$.
\end{center} Remark \ref{remarklevy} and Lemma 4 in~\cite{MR925267} show $P_{\ast}$ is $\lambda$-centered. By the proof of Lemma~\ref{centeredquotient}, we have that ${P \ast \dot{\mathrm{Coll}}(\lambda,<j(\kappa))}$ forces $P_\ast / \dot{G} \ast \dot{H}$ is $\lambda$-centered, as desired.
\end{rema}

We conclude this paper with the following question. 
\begin{ques}
 Is it consistent that there is an ideal over $\aleph_1$ which is $\aleph_2$-saturated but not $(\aleph_2,\aleph_2,2)$-saturated?
\end{ques}
Note that every $\aleph_2$-saturated ideal over $\aleph_1$ is $(\aleph_2,\aleph_1,2)$-saturated under the $\mathrm{CH}$. It is known that $\mathrm{CH}$ implies the partition relation $\aleph_2 \to (\aleph_2,\aleph_1)$ (See Theorem 3,10 in~\cite{MR2768681}). That is, for every $c:[\aleph_2]^{2} \to 2$, if there is no $H_0 \in [\aleph_2]^{\aleph_2}$ with $c ``[H_0]^2 = \{0\}$ then there is an $H_1 \in [\aleph_2]^{\aleph_1}$ with $c ``[H_1]^{2} = \{1\}$. For a $\aleph_2$-saturated ideal $I$ over $\aleph_1$ and $I$-positive sets $\{A_\xi \mid \xi < \aleph_2\}$, define $c:[\aleph_2]^{2} \to 2$ by $c(\xi,\zeta) = 0 \leftrightarrow A_\xi \cap A_{\zeta} \in I$. By the saturation of $I$, there is no $H_0 \in [\aleph_2]^{\aleph_2}$ with $c ``[H_0]^2 = \{0\}$. We can take an $H_1 \in [\aleph_2]^{\aleph_1}$ with $c ``[H_1]^{2} = 1$. For every $\xi < \zeta$ in $H_1$,  by $c(\xi,\zeta) = 1$, $A_\xi \cap A_{\zeta} \not\in I$.

   \bibliographystyle{plain}
   \bibliography{reference}

%\section*{Declarations}
%\subsection*{Competing Interests}
%This research was supported by Grant-in-Aid for JSPS Research Fellow Number 20J21103. 
%\subsection*{Availability of date and materials}
%Not applicable.
\end{document}